\documentclass[final,nomarks]{dmtcs-episciences}

\usepackage[utf8]{inputenc}
\usepackage{subfigure}

\usepackage{tikz}
\usetikzlibrary{shapes,arrows}  
\usetikzlibrary{positioning,chains,fit,shapes,calc}

\newtheorem{prop}{Proposition}
\newtheorem{lem}{Lemma}

\newtheorem{defn}{Definition}
\newtheorem{problem}{Problem}
\newtheorem{theorem}{Theorem}
\usepackage[round]{natbib}

\author{Alex Alochukwu\affiliationmark{1}\thanks{The results of this paper form part of the 
first authors PhD thesis. Financial support by the University of Johannesburg}
  \and Peter Dankelmann\affiliationmark{1}\thanks{Financial support by the South African National Research Foundation, grant 118521, is
gratefully acknowledged.}}
\title[Wiener index in graphs with given min and max degree]{Wiener index in graphs with given minimum degree and 
maximum degree}
\affiliation{University of Johannesburg, South Africa}
\keywords{Wiener index; average distance; mean distance; total distance;  
maximum degree; minimum degree \hspace{0.1cm} MSC-class: 05C12}

\received{2020-12-03}
\revised{2021-04-15}
\accepted{2021-04-19}
\begin{document}
\publicationdetails{23}{2021}{1}{11}{6956}
\maketitle

\begin{abstract}
Let $G$ be a connected graph of order $n$. 
The Wiener index $W(G)$ of $G$ is the sum of the distances between all unordered 
pairs of vertices of $G$. 
In this paper we show that the well-known upper bound  
$\big( \frac{n}{\delta+1}+2\big) {n \choose 2}$ 
on the Wiener index of a graph of order $n$ and minimum degree $\delta$ 
[M.\ Kouider, P.\ Winkler, 
         Mean distance and minimum degree.
         J.\ Graph Theory {\bf 25} no.\ 1 (1997)]
can be improved significantly if the graph contains also a vertex of 
large degree. Specifically, we give the asymptotically sharp bound 
$W(G) \leq {n-\Delta+\delta \choose 2} \frac{n+2\Delta}{\delta+1}+ 2n(n-1)$
on the Wiener index $W(G)$ of a graph $G$ of order $n$, minimum degree
$\delta$ and maximum degree $\Delta$.  
We prove a similar result for triangle-free graphs, and we determine
a bound on the Wiener index of $C_4$-free graphs of given order,
minimum and maximum degree and show that it is, in some sense, 
best possible. 
\end{abstract}

\section{Introduction}
Let $G$ be a connected graph. The \emph{Wiener index} $W(G)$ of $G$ is defined as the 
sum of the distances between all unordered pairs of vertices of $G$, i.e.,
\[ W(G) = \sum_{ \{u,v\} \subseteq V(G)} d_G(u,v), \]
where $V(G)$ denotes the vertex set of $G$ and $d_G(u,v)$ is the usual 
distance between $u$ and $v$. A closely related graph invariant is the 
\emph{average distance} $\mu(G)$ of $G$, defined as 
\[ \mu(G) = {|V(G)| \choose 2}^{-1} \sum_{ \{u,v\} \subseteq V(G)} d_G(u,v).  \]
$W(G)$ and $\mu(G)$ differ by a factor of ${n \choose 2}$ if $n$
is the order of $G$. For graphs of given order, these two invariants are thus
essentially interchangeable. 
The Wiener index was originally introduced as a tool in chemistry, while the 
average distance is 
a useful tool for the analysis of networks in mathematics and computer science 
since it is an indicator for the expected distance between two randomly chosen 
points in a network.  

The natural question which connected graph maximises the Wiener index among all
connected graphs of given order was answered, for example, in a classical paper
by \cite{Ple1984}. We find it convenient to state his result in terms
of average distance.

\begin{prop} \label{prop:path-maximises}
If $G$ is a connected graph with $n$ vertices, then 
\[ \mu(G) \leq \frac{n+1}{3}, \]
with equality if and only if $G$ is a path. 
\end{prop}
 
Much research has been devoted to improving the bound in 
Proposition \ref{prop:path-maximises}. A natural approach 
is to consider not only the order of a graph, but also the degrees 
of the vertices. The \emph{degree} ${\rm deg}_{G}(v)$ of a vertex $v$ in $G$ is the number of vertices adjacent 
to $v$. The \emph{minimum degree} $\delta(G)$ and \emph{maximum degree} 
$\Delta(G)$ is the smallest and largest degree of the vertices in $G$, and 
a $\delta$-\emph{regular} graph is one in which all vertices have degree $\delta$. 

In response to a conjecture by the computer programme GRAFFITI \cite{FajWal1987}, 
which states that in a $\delta$-regular graph of order $n$ 
the average distance is bounded from above by $\frac{n}{\delta}$, \cite{KouWin1997} proved that for every connected 
graph $G$ of order $n$ and minimum degree $\delta$, 
\begin{equation} \label{eq:Kouider-Winkler}
\mu(G) \leq \frac{n}{\delta+1} +2, 
\end{equation}
and this bound is sharp apart from an additive constant. 
Improved bounds for triangle-free and $C_4$-free graphs were given in 
\cite{DanEnt2000}, and for graphs not containing a complete bipartite graph 
$K_{3,3}$ as a subgraph in \cite{DanDlaSwa2005}.  
\cite{BeeRieSmi2001} slightly strengthened  
\ref{eq:Kouider-Winkler} by taking into account also the number of edges, and
thus proved the above-mentioned conjecture.  
In \cite{Dan2006}, the maximum average 
distance of a graph in terms of order, minimum degree and 
size was determined up to an additive constant. 

It is easily shown that the above-mentioned GRAFFITI conjecture 
$\mu \leq \frac{n}{\delta}$ for $\delta$-regular graphs cannot be generalised 
to graphs that are not necessarily regular by replacing $\delta$ 
by the arithmetic mean of the vertex degrees. However, the question whether replacing
$\delta$ by the harmonic mean of the vertex degrees yields a valid inequality, i.e., 
the question if the inequality 
$\mu(G) \leq \sum_{v \in V(G)} \frac{1}{{\rm deg}(v)}$ 
holds for every connected graph was answered in the negative by \cite{ErdPacSpe1988}, who showed that this conjecture
becomes a true inequality if the right hand side is 
multiplied by a term $O(\frac{\log n}{\log \log n})$. 
 
Bounds on the Wiener index of trees in terms of vertex degree have also been 
considered extensively. Every tree has minimum degree $1$, so 
it is natural to ask how large or small
the Wiener index can be in trees of given maximum degree. Answering this question for 
the maximum value of the Wiener index is fairly straightforward (see \cite{Ple1984}
and \cite{Ste2008}), however determining the minimum Wiener index requires more 
effort \cite{FisHofRauSzeVol2002}. For the more general problem of determining the 
extremal values on the Wiener index of a tree with given degree sequence see, for example,  \cite{CelSchWim2Woe2011}, \cite{SchWagWan2012} and \cite{Wan2008}. Trees in which
all vertices have odd degree were considered, for example, in \cite{Lin2013} and \cite{For2013}.

A bound on the Wiener index of the strong product of graphs with 
given minimum degree was  derived in \cite{CasDan2019}. 
Sharp bounds on the Wiener index of maximal $k$-degenerate graphs, i.e., graphs that are 
edge-maximal with the property that every subgraph has a vertex of degree at 
most $k$, were given by \cite{BicChe2019}.
The Wiener index of Eulerian digraphs of given minimum
degree was considered in \cite{Dan2019}. For further questions and results on the 
Wiener index of digraphs we refer the reader to \cite{KnoSkrTep2016}, 
\cite{KnoSkrTep2016-2} and the classical paper \cite{Ple1984}. We mention two further noteworthy recent papers \cite{Muk2014, KlaNad2014} on the Wiener index.

The graphs that demonstrate that \ref{eq:Kouider-Winkler} is sharp apart from an 
additive constant are either regular or close to regular, and this applies also  
to the corresponding
bounds for triangle-free graphs in \cite{DanEnt2000}. This suggests that stronger
bounds hold for graphs of given minimum degree that contain a vertex whose degree 
is significantly larger, for example if the maximum degree is $cn$, where $c\in (0,1)$ 
and $n$ is the order of the graph. The aim of this paper is to show that this is indeed 
the case and to determine upper bounds on the Wiener index or average distance of
graphs in terms of order, minimum degree and maximum degree that are sharp apart
from an additive constant. We note that similar bounds for the average eccentricity
were given in \cite{DanOsa2020}.

This paper is organised as follows. In Section \ref{section:terminology} we define
the most important terms and the notation used in this paper. 
Section \ref{section:weighted-average-distance} defines the weighted versions of 
the average distance and the Wiener index, and gives a result which will be crucial
to the proofs of our main results. In Section \ref{section:bound-for-all-graphs} 
we determine, apart from an additive constant, the maximum Wiener index and 
average distance of a graph of given order, minimum degree and maximum degree.
A corresponding result for triangle-free graphs is given in 
Section \ref{section:triangle-free-graphs}. We consider graphs not containing a 
$4$-cycle as a subgraph in Section \ref{section:C4-fre-graphs}. 
In the concluding Section \ref{section:open-problems} we pose an open
problem arising from our investigations.

\section{Terminology and Notation}
\label{section:terminology}

The notation we use is as follows. 
We denote the vertex set and edge set of a graph $G$ by $V(G)$ and $E(G)$,
respectively, and $n(G)$ stands for the order of $G$, i.e., for 
the number of vertices of $G$. 
The \emph{neighbourhood} of a vertex $v$, i.e., the set of all vertices adjacent to $v$,
is denoted by $N_G(v)$ and its elements are the \emph{neighbours} of $v$. 
The \emph{closed neighbourhood} of $v$ is the set 
$N_G(v) \cup \{v\}$, denoted by $N_G[v]$, and the \emph{second neighbourhood}
of $v$ is the set $\{u \in V(G) \ | \ d_G(v,u) \leq 2\}$, which we denote 
by $N_{\leq 2}(v)$.

If $A\subseteq V(G)$, then we write $G[A]$ for the subgraph of $G$ induced by $A$, 
i.e., the subgraph of $G$ whose vertex set is $A$ and in which two vertices are
adjacent if and only if they are adjacent in $G$. 
The \emph{line graph} of a graph $G$ is the graph $L$ whose vertex set is $E(G)$,
with two vertices of $L$ being adjacent in $L$ if, as edges of $G$, they share 
a vertex.
The $k$-th \emph{power} of $G$, denoted by $G^k$, is the graph with the same vertex set
as $G$ in which two vertices are adjacent if their distance is not more than $k$. 

A \emph{$k$-packing} of $G$ is a set $A$ of vertices in which the distance between
any two vertices of $A$ is greater than $k$. A $2$-packing is often referred to
as a \emph{packing}. 
A \emph{matching} of $G$ is a set of edges in which no two edges share a vertex.
The vertex set $V(M)$ of a matching $M$ is the 
set of vertices incident with an edge in $M$. The \emph{distance} $d_G(e_1, e_2)$ between 
two edges $e_1$ and $e_2$ is the smallest of the distances between a vertex
incident with $e_1$ and a vertex incident with $e_2$. (Note that in general this
is not equal to the distance in the line graph of $G$.) If $M$ is a set of edges,
then the distance $d(e,M)$ between an edge $e$ and $M$ is the smallest of the distances
between $e$ and the edges in $M$. 

A \emph{star} is a tree that has one vertex which is adjacent to all other 
vertices. A \emph{double-star} is a tree that has two vertices such that each of the other
vertices is adjacent to exactly one of them. 
By a {\rm triangle} we mean the complete graph on three vertices, and a graph is 
\emph{triangle-free} if it does not contain a triangle as a subgraph. 
The cycle on $4$-vertices is denoted by $C_4$, and we say a graph is $C_4$-free 
if it does not contain $C_4$ as a subgraph. We note that some authors use the term
$C_4$-free to mean that $C_4$ is not an \emph{induced} subgraph, while we exclude $C_4$
as a subgraph, irrespective whether it is induced or not.

\section{Weighted Wiener index and average distance }
\label{section:weighted-average-distance}

In this section we recall the definitions of weighted versions of the 
Wiener index and the average distance, as well as a bound on the weighted average distance  which is used in the proofs of our main theorems. Both were given in 
\cite{DanEnt2000,KlaGutman1997}.

\begin{defn}
Let $G$ be a connected graph with weight function $c: V(G) \rightarrow \reals^{\geq 0}$. 
Then the Wiener index of $G$ with respect to $c$ is defined as 
\[ W_c(G) = \sum_{\{u,v\} \subseteq V(G)} c(u)c(v) d(u,v). \]
Let $N = \sum_{v \in V(G)}c(v)$. If $N>1$, then we define the average distance
of $G$ with respect to $c$ as 
\[   \mu_c(G) = {N \choose 2}^{-1} \sum_{\{u,v\} \subseteq V(G)} c(u)c(v) d(u,v). \]
\end{defn}

We note that the constant weight function $c$ with $c(v)=1$ for all $v\in V(G)$ 
yields the usual Wiener index and average distance. 

\begin{lem} \label{la:weighted-avdist-maximised-by-path}
Let $G$ be a connected graph with weight function $c: V(G) \rightarrow \reals^{\geq 0}$ 
and let $k,N$ be two positive integers, $N$ a multiple of $k$, such that $c(v) \geq k$ 
for every vertex $v$ of $G$ and $\sum_{v \in V(G)} c(v) \leq N$. Then 
\[ \mu_c(G) \leq \frac{N-k}{N-1} \frac{N+k}{3k}. \]
Equality holds if and only if $G$ is a path and $c(v)=k$ for every $v\in V$. 
\end{lem}
  
\section{A bound involving the maximum degree}
\label{section:bound-for-all-graphs}

In this section we prove our first main result, an upper bound on the Wiener index 
of graphs of given order, minimum degree and maximum degree.  We state
our result in terms of the average distance and not the Wiener index as this
is more convenient for the proof. 

\begin{theorem} \label{theo:maxdegree}
Let $G$ be a connected graph of order $n$, minimum degree $\delta \geq 3$ and 
maximum degree $\Delta$. Then 
\[ \mu(G) \leq 
\frac{N(N-1)}{n(n-1)} \frac{n+2\Delta}{\delta+1}  +4, \]
where $N=n-\Delta+\delta$,  
and this bound is sharp apart from an additive constant.  
\end{theorem}

\begin{proof}
We first find a packing $A$ of $G$ as follows. 
Choose a vertex $v_1$ of degree $\Delta$ in $G$ and let  $A= \{v_1\}$. 
If there exists a vertex $v_2$ of $G$ at distance exactly $3$ from $A$, then add $v_2$
to $A$ and let $A=\{v_1, v_2\}$.  If there exists a vertex $v_3$ at distance exactly $3$  
from $A=\{v_1,v_2\}$, then add also $v_3$ to $A$. 
Repeat this step, i.e., successively add vertices at distance $3$ from $A$  
until, after $k$ steps say, each vertex of $G$ is within distance $2$ of $A$.  
Then $A=\{v_1, v_2,\ldots, v_k\}$. 

By the construction of $A$, the sets $N[v_i]$ are pairwise disjoint for $i=1,2,\ldots,k$. 
For $i=1,2,\ldots, k$ let $T(v_i)$ be the star with vertex set $N[v_i]$ centred 
at $v_i$. The disjoint union $\bigcup_{i=1}^k T(v_i)$ forms a subforest 
$T_1$ of $G$.  
It follows from the construction of $A$ that for every $i \in \{2,3,\ldots,k\}$
there exists an edge $e_i$ in $G$ joining a vertex in $T(v_i)$ to a vertex in $T(v_j)$
for some $j$ with $1 \leq j <i$.  Hence $T_2:=T_1 + \{e_2, e_3,\ldots, e_k\}$ is a subtree
of $G$. By the maximality of $A$ every vertex of $G$ not in $T_2$ is adjacent to 
some vertex of $T_2$.  
Hence we can extend $T_2$ to a spanning tree of $G$ by adding each vertex
in $V(G)-V(T_2)$ to $T_2$ and joining it to some vertex of $T_2$ with an edge of $G$. 
Denote the resulting spanning tree of $G$ by $T$.

Since the average distance of any spanning tree of $G$ is not less than 
the average distance of $G$, it suffices to show that
\begin{equation} \label{eq:bound-on-mu(T)} 
\mu(T) \leq 
\frac{n-\Delta+\delta}{n} \frac{n-\Delta+\delta-1}{n-1} 
       \frac{n+2\Delta}{\delta+1} + 4.
\end{equation}  
The tree $T$ can be thought of as a weighted tree, where each vertex has weight exactly $1$. 
Informally speaking, we now move the weight of every vertex to the closest vertex in 
$A$. More precisely, 
for every vertex $u \in V (T)$ let $u_A$ be a vertex in $A$ closest to $u$ and  
define a weight function $c: V(T) \rightarrow \reals^{\geq 0}$ by 
\[ c(v) = | \{u \in V(T) \ | \ u_A = v\}|. \]
If $v\in A$, then every $u\in N[v]$ satisfies $u_A=v$, hence 
$c(v)\geq {\rm deg}_G(v)+1$. Thus $c(v_1)\geq \Delta+1$ and 
$c(v_i) \geq \delta+1$ if $i \in \{2,3,\ldots,k\}$. 
If $v\in V(G)-A$, then $c(v)=0$. 
We also note that $\sum_{v \in V(T)} c(v)=n$. This yields 
$n = \sum_{v \in A}c(v) \geq \Delta+1 + (k-1)(\delta+1)$, and so
\begin{equation} \label{eq:upper-bound-on-k}
k \leq \frac{n-\Delta+ \delta}{\delta+1}. 
\end{equation}
In our proof we consider $\mu_c(T)$ rather than $\mu(T)$. Informally, the difference
between these two values is not more than $4$ since no weight was moved over a 
distance of more than $2$, so no distance between two weights has changed
by more than $4$. More formally, 
\begin{eqnarray*}
\mu(T) & = & {n \choose 2}^{-1} 
   \sum_{ \{u,v\} \subseteq V(T)} d_T(u,v)  \\
  & = & {n \choose 2}^{-1} 
  \Big( \sum_{ \{u,v\} \subseteq V(T)} d_T(u_A,v_A) 
       + \sum_{ \{u,v\} \subseteq V(T)} \big(d_T(u,v) - d_T(u_A,v_A)\big) \Big) \\
  & = & \mu_c(T) + {n \choose 2}^{-1} \sum_{ \{u,v\} \subseteq V(T)} \Big(d_T(u,v) - d_T(u_A,v_A)\Big), 
\end{eqnarray*}
and thus, since $d_T(u,v)-d_T(u_A,v_A) \leq d_T(u,u_A) + d_T(v,v_A) \leq 4$ 
for all $u,v \in V(T)$, we have 
\begin{equation} \label{eq:mu-mu_c}
\mu(T) \leq \mu_c(T) + 4.   
\end{equation}
We note that only the vertices in $A$ have a nonzero weight. By the construction
of $T$, each vertex $v_i \in A$, $i\geq 2$, is at distance exactly three in $T$ from 
some vertex $v_j \in A$ with $j<i$. This implies that in the graph $T^3[A]$ there is 
a path from every vertex to $v_1$, and so $T^3[A]$ is connected. Let $H=T^3[A]$. 
Then $d_T(u,v) \leq 3 d_{H}(u,v)$ for all $u,v \in A$, and so 
\begin{equation} \label{eq:muT3[A]-vs-muT}
\mu_c(T) \leq 3 \mu_c(H).
\end{equation}  
We modify the weight function $c$ to obtain a new weight function $c'$ on $A$ 
which satisfies $c'(v) \geq \delta+1$ for all $v\in A$. We let 
\[ c'(v) = \left\{ \begin{array}{cc}
c(v) -\Delta + \delta & \textrm{if $v=v_1$,} \\
c(v)                  & \textrm{if $v \in A-\{v_1\}$}. 
           \end{array} \right.  \]
Since ${\rm deg}_G(v)=\Delta$ if $v=v_1$ and ${\rm deg}_G(v) \geq \delta$ if 
$v\in A-\{v_1\}$, we have $c'(v) \geq \delta+1$ for all $v \in A$. 
Let $N:= \sum_{v \in A} c'(v)$, so $N = n-\Delta+\delta$. 

We now express $\mu_c(H)$ in terms of $\mu_{c'}(H)$. We have 
\begin{eqnarray*}
{N \choose 2} \mu_{c'}(H) & = &  W_{c'}(H) \\
  & = &  \sum_{\{u,v\} \subseteq A-\{v_1\}} c(u)c(v)d_H(u,v) \\
 & &  + 
   \sum_{u \in \subseteq A-\{v_1\}} (c(v_1)-\Delta+\delta)c(u)d_H(u,v_1) \\
  & = &  \sum_{\{u,v\} \subseteq A} c(u)c(v)d_H(u,v) \\
 & &  - 
  (\Delta-\delta) \sum_{u \in \subseteq A-\{v_1\}} c(u)d_H(u,v_1) \\
 & = &    W_c(H) - 
       (\Delta-\delta) \sum_{u \in A-\{v_1\}} c(u)d_H(u,v_1).  
\end{eqnarray*} 
Expressing $W_c(H)$ as ${n \choose 2} \mu_c(H)$ and rearranging yields
\begin{equation}  \label{eq:mu_c-in-term-of-mu_c'} 
\mu_c(H) = \frac{N(N-1)}{n(n-1)} \mu_{c'}(H)  
             + \frac{2(\Delta-\delta)}{n(n-1)} \sum_{u \in A-\{v_1\}} c(u)d_H(u,v_1). 
\end{equation}             
We bound the two terms of the right hand side of 
\ref{eq:mu_c-in-term-of-mu_c'} separately. Renaming the vertices of
$A-\{v_1\}$ as $u_1, u_2,\ldots,u_{k-1}$ such that 
$d_H(v_1,u_1) \leq d_H(v_1,u_2) \leq \ldots \leq d_H(v_1,u_{k-1})$, we 
have $d_H(v_1,u_i) \leq i$ and so 
\begin{equation} \label{eq:weighted-distance-of-v1}
 \sum_{u \in A-\{v_1\}} c(u)d_H(u,v_1) \leq \sum_{i=1}^{k-1}  c(u_i)\, i. 
\end{equation}
Now $c(u_i) \geq \delta+1$ for $i=1,2,\ldots,k-2$ and 
$\sum_{i=1}^{k-1} c(u_i) = n- c(v_1) \leq n- \Delta-1$. Subject to these 
conditions the right hand side of \ref{eq:weighted-distance-of-v1} is maximised
if $c(u_1)=c(u_2)=\ldots=c(u_{k-2}) = \delta+1$ and 
$c(u_{k-1})=n-\Delta-1 - (k-2)(\delta+1)$. Substituting these values yields, 
after simplification, 
\begin{equation} \label{eq:weighted-distance-of-v1-2}
\sum_{i=1}^{k-1} c(u_i)\, i \leq (n-\Delta-1)(k-1) - \frac{1}{2}(\delta+1)(k-1)(k-2). 
\end{equation}
Since $k \leq \frac{n-\Delta+\delta}{\delta+1}$ by \ref{eq:upper-bound-on-k}, 
and since the right hand side is increasing in $k$ for 
$k \leq \frac{n-\Delta+\delta}{\delta+1}$, we obtain by substituting this value 
and combining \ref{eq:weighted-distance-of-v1} and 
\ref{eq:weighted-distance-of-v1-2} that 
\[
 \sum_{u \in A-\{v_1\}} c(u)d_H(u,v_1) 
   \leq \frac{(n-\Delta-1)^2}{2(\delta+1)} + \frac{n-\Delta-1}{2}.
\]
Since 
$\frac{(n-\Delta-1)^2}{2(\delta+1)} + \frac{n-\Delta-1}{2} 
  =\frac{(n-\Delta-1)(n-\Delta+\delta)}{2(\delta+1)} 
  \leq \frac{N(N-1)}{2(\delta+1)}$, this implies  that 
\begin{equation} \label{eq:weighted-distance-of-v1-3}
\sum_{u \in A-\{v_1\}} c(u)d_H(u,v_1) 
   \leq \frac{(N-1)N}{2(\delta+1)}.
\end{equation}
To bound $\mu_{c'}(H)$ we recall that $c'(v) \geq \delta+1$ for all 
vertices $v$ of $H$ and that  $\sum_{v \in A} c'(v)=n-\Delta+\delta$.  
Let $C$ be the least multiple of $\delta+1$ with $C\geq N$. 
By Lemma \ref{la:weighted-avdist-maximised-by-path} we have 
\[
\mu_{c'}(H) \leq \frac{C-\delta-1}{C-1} \frac{C+\delta+1}{3(\delta+1)} 
      \leq \frac{C+1}{3(\delta+1)}.
\]
Since $C\leq N+\delta$, this yields 
\begin{equation}  \label{eq:bound-on-mu_c'(H)} 
\mu_{c'}(H)  \leq \frac{N+\delta+1}{3(\delta+1)}.
\end{equation}
Substituting \ref{eq:weighted-distance-of-v1-3} and \ref{eq:bound-on-mu_c'(H)} 
into \ref{eq:mu_c-in-term-of-mu_c'} yields 
\begin{eqnarray}
\mu_{c}(H) & \leq & \frac{N(N-1)}{n(n-1)} \frac{N+\delta+1}{3(\delta+1)}  
        + \frac{2(\Delta-\delta)}{n(n-1)} \frac{N(N-1)}{2(\delta+1)} \nonumber \\
  & \leq & \frac{N(N-1)}{n(n-1)} \frac{n+2\Delta}{3(\delta+1)}. \label{eq:total-bound}      
\end{eqnarray}
Combining the inequalities  \ref{eq:mu-mu_c}, \ref{eq:muT3[A]-vs-muT} 
and  \ref{eq:total-bound} we obtain that 
\begin{eqnarray*}
\mu(T) & \leq & \mu_c(T) + 4 \\
 & \leq & 3 \mu_c(H) + 4 \\
 & \leq & \frac{N(N-1)}{n(n-1)} \frac{n+2\Delta}{\delta+1}  +4, 
\end{eqnarray*}
which is \ref{eq:bound-on-mu(T)}, and so the bound in the theorem follows. \\

To show that the above bound is sharp apart from
an additive constant we consider the following graph for given 
$n, \Delta, \delta \in \naturals$ for which $\Delta \geq \delta$ and 
$n \geq \Delta+\delta+1$. In order to keep our calculations 
simple we assume that $n-\Delta$ is a multiple of $\delta+1$, but we note that 
a similar construction, with slightly more labourious calculations, can be done
without this additional  assumption. 
Let $\ell = \frac{n-\Delta+\delta+1}{\delta+1}$. 
Let $G_1$ be a copy of the complete graph $K_{\Delta}$, let $G_2, G_3,\ldots,G_{\ell}$ 
be disjoint copies of $K_{\delta+1}$.  
For $i=1,2,\ldots,\ell$ let $u_i$ and $v_i$ be two vertices of $G_i$. Define
$G_{n,\Delta,\delta}$ to be the graph obtained from $\bigcup_{i=1}^{\ell} G_i$ 
by deleting the edge $u_iv_i$ for $i=2,3,\ldots,\ell-1$ and adding the edge 
$v_iu_{i+1}$ for $i=1,2,\ldots,\ell-1$. The graph $G_{28,8,3}$ is shown 
in Figure \ref{fig:sharpness-ex}.

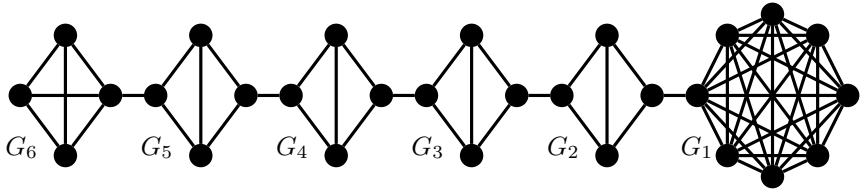
\begin{figure}
\begin{center}
\begin{tikzpicture}
  [scale=0.4,inner sep=1mm, 
   vertex/.style={circle,thick,draw}, 
   thickedge/.style={line width=2pt}] 
    \begin{scope}[>=triangle 45]
    \node[vertex] (a1) at (0,0) [fill=black] {};
    \node[vertex] (a3) at (3,0) [fill=black] {};
    \node[vertex] (a4) at (4.5,0) [fill=black] {};
    \node[vertex] (a6) at (7.5,0) [fill=black] {};
    \node[vertex] (a7) at (9,0) [fill=black] {};
    \node[vertex] (a9) at (12,0) [fill=black] {};
    \node[vertex] (a10) at (13.5,0) [fill=black] {};
    \node[vertex] (a12) at (16.5,0) [fill=black] {};
  \node[vertex] (a13) at (18,0) [fill=black] {};
   \node[vertex] (a14) at (21,0) [fill=black] {};
   \node[vertex] (a15) at (22.5,0) [fill=black] {};
   \node[vertex] (a16) at (27.5,0) [fill=black] {};
    \node[vertex] (b1) at (1.5,2) [fill=black] {};
    \node[vertex] (b2) at (6,2) [fill=black] {};    
    \node[vertex] (b3) at (10.5,2) [fill=black] {};
    \node[vertex] (b4) at (15,2) [fill=black] {};
     \node[vertex] (b5) at (19.5,2) [fill=black] {};
     \node[vertex] (b6) at (25,2.7) [fill=black] {};
   \node[vertex] (c6) at (25,-2.7) [fill=black] {};    
   
   \node[vertex] (b7) at (23.5,2) [fill=black] {};
   \node[vertex] (c7) at (23.5,-2) [fill=black] {};
   \node[vertex] (b8) at (26.5,2) [fill=black] {};
   \node[vertex] (c8) at (26.5,-2) [fill=black] {};
      \node[vertex](c1) at (1.5,-2) [fill=black] {};
    \node[vertex] (c2) at (6,-2) [fill=black] {};    
    \node[vertex] (c3) at (10.5,-2) [fill=black] {};
    \node[vertex] (c4) at (15,-2) [fill=black] {};
     \node[vertex] (c5) at (19.5,-2) [fill=black] {};    
    \draw[very thick] (a1)--(a3);  
    \draw[very thick] (b1)--(a3);  
    \draw[very thick] (a3)--(a4);  
    \draw[very thick] (b1)--(a1);  
    \draw[very thick] (b2)--(a6);  
    \draw[very thick] (a6)--(a7);  
    \draw[very thick] (b2)--(a4);  
   \draw[very thick] (b3)--(a9);   
    \draw[very thick] (a9)--(a10);  
    \draw[very thick] (b3)--(a7);  
    \draw[very thick] (b4)--(a10);       
    \draw[very thick] (b4)--(a12);
    \draw[very thick] (a12)--(a13);
    \draw[very thick] (a14)--(a15);
    \draw[very thick] (a15)--(a16);
    \draw[very thick] (b5)--(a13);
    \draw[very thick] (b5)--(a14);      
    \draw[very thick] (c1)--(a3);
    \draw[very thick] (c1)--(a1);
    \draw[very thick] (c2)--(a4);  

    \draw[very thick] (c2)--(a6);  
    \draw[very thick] (c3)--(a7);  
    \draw[very thick] (c3)--(a9);  
    \draw[very thick] (c4)--(a10);  

    \draw[very thick] (c4)--(a12);
    \draw[very thick] (c5)--(a13);    
    \draw[very thick] (c5)--(a14);

    \draw[very thick] (b1)--(c1);        
    \draw[very thick] (b2)--(c2);        
    \draw[very thick] (b3)--(c3);        
    \draw[very thick] (b4)--(c4);        
    \draw[very thick] (b5)--(c5); 
     
    \draw[very thick] (a15)--(b7);  
    \draw[very thick] (a15)--(c7);
    \draw[very thick] (a15)--(b6);
    \draw[very thick] (a15)--(c6);
    \draw[very thick] (a15)--(b8);
    \draw[very thick] (a15)--(c8); 
    \draw[very thick] (a16)--(b8);
    \draw[very thick] (a16)--(c8);
    \draw[very thick] (b7)--(b8);
    \draw[very thick] (c7)--(c8);
     \draw[very thick] (a16)--(b7);  
    \draw[very thick] (a16)--(c7);
    \draw[very thick] (a16)--(b6);
    \draw[very thick] (a16)--(c6);
     \draw[very thick] (b6)--(b7);  
    \draw[very thick] (b6)--(b8);
    \draw[very thick] (b6)--(c7);
    \draw[very thick] (b6)--(c8);
     \draw[very thick] (b6)--(c6);  
    \draw[very thick] (b7)--(c7);
    \draw[very thick] (b7)--(c6);
    \draw[very thick] (b7)--(c8);
     \draw[very thick] (b8)--(c7);  
    \draw[very thick] (b8)--(c6);
    \draw[very thick] (b8)--(c8);
     \draw[very thick] (c7)--(c6);  
    \draw[very thick] (c8)--(c6);
   
     \node[below] at (0.05,-1.1) {$G_{6}$}; 
    \node[below] at (4.55,-1.1) {$G_{5}$}; 
    \node[below] at (9.05,-1.1) {$G_{4}$}; 
    \node[below] at (13.55,-1.1) {$G_{3}$};           
    \node[below] at (18.05,-1.1) {$G_{2}$};            
     \node[below] at (22.5,-1.1) {$G_{1}$};            
    \end{scope}
\end{tikzpicture}
\end{center}
\caption{The graph $G_{28,8,3}$}\label{fig:sharpness-ex}
\end{figure}

It is easy to verify that $G_{n,\Delta,\delta}$ is a 
connected graph of order $n$, maximum degree $\Delta$ and minimum degree $\delta$. 

We now bound the average distance of $G_{n,\Delta,\delta}$ from below. 
For $i=1,2,\ldots,\ell$ let $V_i=V(G_i)$.
Then $|V_i|=\delta+1$ for $i=2,3,\ldots,\ell$, and 
$|V_1|=\Delta$. For our lower bound we only count the distances
between pairs $x,y$ with either $x\in V_i$ and $y\in V_j$ where $2\leq i<j\leq \ell$, or
$x \in V_1$ and $y \in V_j$ for $j=2,3,\ldots,\ell$, ignoring all
other pairs of vertices. Clearly, if $x \in V_i$, $y\in V_j$ and $i<j$, then 
$d(x,y) \geq 3(j-i)-2$. If $x \in V_1$ and $y \in V_j$, then 
$d(x,y) \geq 3j-5$. Hence 
\begin{eqnarray*}
W(G_{n,\Delta,\delta}) & > & \sum_{2\leq i < j \leq \ell} 
           \sum_{x\in V_i, y \in V_j} d(x,y) + 
           \sum_{j=2}^{\ell} \sum_{x\in V_1, y \in V_j} d(x,y) \\
    & \geq &  \sum_{2\leq i < j \leq \ell} |V_i| \, |V_j|  \big( 3(j-i)-2 \big) 
          + \sum_{j=2}^{\ell} |V_1| \, |V_j| \big( 3j-5\big) \\
     & = & (\delta+1)^2 \sum_{2\leq i < j \leq \ell}  \big( 3(j-i)-2 \big) 
          + (\delta+1) \Delta \sum_{j=2}^{\ell}  \big( 3j-5\big).
\end{eqnarray*}
Straightforward calculations show that 
$\sum_{2\leq i < j \leq \ell}  \big( 3(j-i)-2 \big)  
   = \frac{1}{2}(\ell^3 - 5\ell^2+8\ell-4) > \frac{1}{2}(\ell^3 - 5\ell^2)$  
and 
$\sum_{j=2}^{\ell}  \big( 3j-5\big) = \frac{1}{2}(3\ell^2-7\ell+4) 
      > \frac{1}{2}(3\ell^2 - 7\ell)$. 
Substituting these values, we obtain
\begin{eqnarray*}
W(G_{n,\Delta,\delta}) & > &  
   \frac{1}{2}(\delta+1)^2(\ell^3-5\ell^2) 
   + \frac{1}{2}(\delta+1)\Delta(3\ell^2-7\ell)  \\
   & = & \frac{1}{2}\big( (\delta+1)^2 \ell^3 + 3(\delta+1)\Delta \ell^2 \big)
         - \frac{5}{2} (\delta+1)^2 \ell^2 - \frac{7}{2}(\delta+1)\Delta\ell.
\end{eqnarray*}
Let $N=n-\Delta+\delta$. Then $\ell(\delta+1) = N+1$.  Therefore, 
$(\delta+1)^2 \ell^3 + 3(\delta+1)\Delta \ell^2 
 = \frac{(N+1)^3}{\delta+1} + \frac{3\Delta(N+1)^2}{\delta+1}
 > \frac{N(N-1)(n+2\Delta)}{\delta+1}$. 
Also 
$ \frac{5}{2}(\delta+1)^2\ell^2 + \frac{7}{2}(\delta+1)\Delta\ell 
 <\frac{7}{2} (N+1)((\delta+1)\ell+ \Delta) =\frac{7}{2}(N+1)(n+\delta+1) < 7n(n-1)$. 
Applying these inequalities to the bound on  $W(G_{n,\Delta,\delta})$ and 
dividing by ${n \choose 2}$ thus yields 
\[ \mu(G_{n,\Delta,\delta}) > \frac{N(N-1)}{n(n-1)} \frac{n+2\Delta}{\delta+1}-14, \]
which differs from the upper bound by at most $18$. 
\end{proof}

Define $\mu(n,\Delta,\delta)$ and $W(n,\Delta,\delta)$ to be the maximum
average distance and 
maximum Wiener index, respectively, among all connected graphs of order $n$, minimum 
degree $\delta$ and maximum degree $\Delta$.
Theorem \ref{theo:maxdegree} shows that 
\[ \mu(n,\Delta,\delta) = \frac{(n-\Delta+\delta)(n-\Delta+\delta-1)}{n(n-1)} \frac{n+2\Delta}{\delta+1} +O(1) \]
and
\[ W(n,\Delta,\delta) 
 = {n-\Delta+\delta \choose 2} \frac{n+2\Delta}{\delta+1}  +O(n^2).  \]

\section{A bound for triangle-free graphs}
\label{section:triangle-free-graphs}

In this section we show that the bound in Theorem \ref{theo:maxdegree} can be 
improved by a factor of about $\frac{2}{3}$ for graphs not containing a
triangle. 

\begin{theorem} \label{theo:maxdegree-triangle-free}
Let $G$ be a connected, triangle-free graph of order $n$, minimum degree $\delta \geq 3$ and 
maximum degree $\Delta$. Then  
\[ \mu(G) \leq \frac{2}{3} \frac{N(N-1)}{n(n-1)} \frac{n+2\Delta}{\delta}  +7, \]
where $N=n-\Delta+\delta$, and this bound is sharp apart from an additive constant.  
\end{theorem}

\begin{proof}
We first find a matching $M$ of $G$ as follows. 
Choose an edge $e_1$ of $G$ that is incident with a vertex
of degree $\Delta$ and let $M= \{e_1\}$. 
If there exists an edge $e_2$ of $G$ at distance exactly $3$ from $M$, then 
add $e_2$ to $M$, i.e., let $M=\{e_1, e_2\}$.  If there exists an edge $e_3$ at distance 
$3$ from $M$ then add also $e_3$ to $M$.  
Repeat this step, i.e., successively add edges at distance $3$ from $M$  
until, after $k$ steps say, each edge of $G$ is within distance $2$ of $M$.  
Let $M=\{e_1, e_2,\ldots, e_k\}$. 
 
For $i=1,2,\ldots, k$ let $e_i=u_iv_i$. Let $T(e_i)$ be the double-star with vertex set 
$N_G(u_i) \cup N_G(v_i)$ that 
contains $e_i$ and all edges of $G$ incident with $u_i$ or $v_i$. 
By the construction of $M$, the trees $T(e_i)$, $i=1,2,\ldots,k$, are disjoint. 
The  union $\bigcup_{i=1}^k T(e_i)$ forms a subforest $T_1$ of $G$.  
It follows from the construction of $M$ that for every $i \in \{2,3,\ldots,k\}$
there exists an edge $f_i$ in $G$ joining a vertex in $T(e_i)$ to a vertex in $T(e_j)$
for some $j$ with $1 \leq j <i$.  Then $T_2:=T_1 + \{f_2, f_3,\ldots, f_k\}$ is a subtree
of $G$. By the maximality of $M$, every edge of $G$ is within distance $2$ from
some edge of $M$, and so every vertex of $G$ is within distance $3$ from 
some vertex of $V(M)$. 
We extend $T_2$ to a spanning tree $T$ of $G$ that preserves the distances 
to a nearest vertex in $V(M)$.

Since the average distance of a spanning tree of $G$ is not less than 
the average distance of $G$, it suffices to show that
\begin{equation} \label{eq:bound-on-mu(T)-trianglefree} 
\mu(T) \leq 
 \frac{2}{3} \frac{(n-\Delta+\delta)(n-\Delta+\delta-1)}{n(n-1)} 
  \frac{n+2\Delta}{\delta}  +7. 
\end{equation}  
For every vertex $u \in V (T)$ let $u_M$ be a vertex in $V(M)$ closest to $u$ in $T$. 
As in Theorem \ref{theo:maxdegree} we define a function 
$c: V(T) \rightarrow \reals^{\geq 0}$ by 
\[ c(v) = | \{u \in V(T) \ | \ u_M = v\}|. \]
As in Theorem \ref{theo:maxdegree}, the weight of $c$ is concentrated in
the vertices in $V(M)$. Since no weight has been moved over a distance of 
more than $3$, a proof similar to the proof
of \ref{eq:mu-mu_c} yields that 
\begin{equation} \label{eq:mu-mu_c-trianglefree}
\mu(T) \leq \mu_c(T) +6. 
\end{equation}
We now consider the line graph $L$ of $T$ and define a new weight function 
$\overline{c}$ on the vertices of $L$, i.e., the edges of $T$, by
\[ \overline{c}(uv) = \left\{ \begin{array}{cc}
 c(u) + c(v) & \textrm{if $uv\in M$,} \\
 0 & \textrm{if $uv \notin M$}. 
\end{array} \right. \]
The weight of $\overline{c}$ is concentrated entirely in $M$, and 
$\sum_{e \in E(T)} \overline{c}(e) =n$.
Consider $e_i=u_iv_i \in M$. Since $G$ is triangle-free, 
$u_i$ and $v_i$ have no common neighbours and so the vertices in 
$N[u_i] - \{v_i\}$ contribute to $c(u_i)$, and the vertices in
$N[v_i] - \{u_i\}$ contribute to $c(v_i)$, implying that
$c(u_i) \geq {\rm deg}_G(u_i)$ and $c(v_i) \geq {\rm deg}_G(v_i)$. 
Since $e_1$ is incident with a vertex of degree $\Delta$, while all the 
other vertices in $V(M)$ have degree at least $\delta$, we conclude that     
\begin{equation}   \label{eq:weight-of-e_i}
\overline{c}(e_i) \geq \left\{ \begin{array}{cc} 
    \Delta + \delta & \textrm{if $i=1$,} \\
    2\delta & \textrm{if $i \in \{2, 3,\ldots, k\}$.}
         \end{array} \right. 
\end{equation}
Clearly $\sum_{e \in E(T)} \overline{c}(e) =n$, and so we have 
\[
n \geq  \sum_{i=1}^k \overline{c}(e_i) 
  \geq \Delta + \delta + 2\delta (k-1)
\]
It follows that
\begin{equation} \label{eq:k-in-trianglefree} 
k \leq \frac{n-\Delta+\delta}{2\delta}. 
\end{equation}
We now bound the difference between $\mu_c(T)$ and $\mu_{\overline{c}}(L)$. 
If $u$ and $v$ are vertices of $T$, and $e_u$ and $e_v$ are edges of $T$ 
incident with $u$ and $v$, respectively, then it is easy to see that
$d_T(u,v) \leq  d_L(e_u,e_v) + 1$. Hence no distance between two weights 
decreases by more than $1$ if we consider $L$ and $\overline{c}$ 
instead of $T$ and $c$. Therefore, 
\begin{equation} \label{eq:mu_c(T)-vs-L}
\mu_c(T) \leq \mu_{\overline{c}}(L) +1. 
\end{equation}
Consider the induced subgraph $L^4[M]$. It follows from the construction
of $M$ that every $e_i \in M$ with $i\geq 2$ is at distance $4$ in $L$ from
some edge $e_j \in M$ with $1\leq j<i$, and so $L^4[M]$ is connected. 
Let $H=L^4[M]$, so  $V(H)=M$. If $e_i, e_j \in M$, then clearly 
$d_L(e_i,e_j) \leq 4d_H(e_i.e_j)$, and so 
\begin{equation} \label{eq:mu(L)-vs-mu(H)-trianglefree} 
\mu_{\overline{c}}(L) \leq 4\mu_{\overline{c}}(H). 
\end{equation}
We now modify the weight function $\overline{c}$ to obtain a new weight function $c'$ 
which satisfies $c'(e) \geq 2\delta$ for all $e\in M$.  We let 
\[ c'(e_i) = \left\{ \begin{array}{cc}
\overline{c}(e_1) -\Delta + \delta & \textrm{if $i=1$,} \\
\overline{c}(e_i)                  & \textrm{if $i \in \{2,3,\ldots,k \}$}. 
           \end{array} \right.  \]
It follows from \ref{eq:weight-of-e_i} that 
$c'(e) \geq 2\delta$ for all $e \in M$. 
Letting $N=\sum_{e \in M} c'(e)$, we have $N = n-\Delta+\delta$. 

We now express $\mu_{\overline{c}}(H)$ in terms of $\mu_{c'}(H)$. 
Arguments similar to those in the proof of \ref{eq:mu_c-in-term-of-mu_c'} 
in the proof of Theorem \ref{theo:maxdegree} show that 
\begin{equation}  \label{eq:mu_c-in-term-of-mu_c'-trianglefree} 
\mu_{\overline{c}}(H) = \frac{N(N-1)}{n(n-1)} \mu_{c'}(H)  
   + \frac{2(\Delta-\delta)}{n(n-1)} \sum_{i=2}^k \overline{c}(e_i)d_H(e_1,e_i). 
\end{equation}             
We bound the two terms of the right hand side of 
\ref{eq:mu_c-in-term-of-mu_c'-trianglefree} separately. 
Arguments similar to the ones in the proof of \ref{eq:weighted-distance-of-v1-2}
show that  
\begin{equation} \label{eq:weighted-distance-of-e1-trianglefree-2}
 \sum_{i=2}^k \overline{c}(e_i)d_H(e_1,e_i) 
        \leq (n-\Delta-\delta)(k-1) - \delta(k-1)(k-2). 
\end{equation}
Since $k \leq \frac{n-\Delta+\delta}{2\delta}$ by \ref{eq:k-in-trianglefree}, 
and since the right hand side is increasing in $k$ for 
$k \leq \frac{n-\Delta+\delta}{2\delta}$, we obtain by substituting this value 
 that 
\[
\sum_{i=2}^k \overline{c}(e_i)d_H(e_1,e_i)  
   \leq \frac{(n-\Delta-\delta)^2}{4\delta} + \frac{n-\Delta-\delta}{2}.
\]
Since $N=n-\Delta+\delta$, the right hand side equals 
$\frac{(N-2\delta)^2}{4\delta} + \frac{N-2\delta}{2} = \frac{N(N-2\delta)}{4\delta}$,
and so we obtain 
\begin{equation} \label{eq:weighted-distance-of-e1-trianglefree-3}
\sum_{i=2}^k \overline{c}(e_i)d_H(e_1,e_i) \leq \frac{N(N-2\delta)}{4\delta}  
           < \frac{N(N-1)}{4\delta}. 
\end{equation}
To bound $\mu_{c'}(H)$ we recall that $c'(e) \geq 2\delta$ for all 
$e\in V(H)=M$ and thus for every vertex $e$ of $H$, and that  
$\sum_{e \in M} c'(e)=N$.  
Let $C$ be the least multiple of $2\delta$ with $C\geq N$. 
By Lemma \ref{la:weighted-avdist-maximised-by-path} we have 
\[  \mu_{c'}(H) \leq \frac{C-2\delta}{C-1} \frac{C+2\delta}{6\delta} 
      \leq \frac{C+1}{6\delta}.    \]
Since $C\leq N+ 2\delta-1$, this yields 
\begin{equation}  \label{eq:bound-on-mu_c'(H)-trianglefree} 
\mu_{c'}(H)  \leq \frac{N+2\delta}{6\delta}.
\end{equation}
Substituting \ref{eq:weighted-distance-of-e1-trianglefree-3} and 
\ref{eq:bound-on-mu_c'(H)-trianglefree} into \ref{eq:mu_c-in-term-of-mu_c'-trianglefree}
  yields 
\begin{eqnarray}  
\mu_{\overline{c}}(H) 
    & \leq & \frac{N(N-1)}{n(n-1)} \frac{N+2\delta}{6\delta}   
          + \frac{2(\Delta-\delta)}{n(n-1)} \frac{N(N-1)}{4\delta} \nonumber \\
   & = & \frac{N(N-1)}{n(n-1)} \frac{n+2\Delta}{6\delta}.   \label{eq:mu_c-bar-trianglefree} 
\end{eqnarray}   

Combining the inequalities  \ref{eq:mu-mu_c-trianglefree}, \ref{eq:mu_c(T)-vs-L},
\ref{eq:mu(L)-vs-mu(H)-trianglefree} 
and  \ref{eq:mu_c-bar-trianglefree} we obtain that 
\begin{eqnarray*}
\mu(T) & \leq & \mu_c(T) + 6 \\
   & \leq & \mu_{\overline{c}}(L) + 7 \\
 & \leq & 4 \mu_{\overline{c}}(H) + 7 \\
 & \leq & \frac{2}{3} \frac{N(N-1)}{n(n-1)} \frac{n+2\Delta}{\delta}  +7, 
\end{eqnarray*}
which is \ref{eq:bound-on-mu(T)-trianglefree}, and so the bound in the theorem follows. \\

To see that the bound is sharp apart from an additive constant let 
$n, \Delta, \delta \in \naturals$ be given with $\Delta \geq \delta$ and 
$n\geq \Delta+3\delta$. 
In order to keep our construction and calculations simple, 
we assume that $n-\Delta-\delta$ is a multiple of $2\delta$, but as in 
Theorem \ref{theo:maxdegree} this assumption is not essential.  
Let $\ell = \frac{n-\Delta+\delta}{2\delta}$. 
Let $G_1$ be the complete bipartite graph $K_{\delta, \Delta}$, and for 
$i=2,3,\ldots,\ell$ let $G_i$ be a copy of the complete bipartite graph 
$K_{\delta,\delta}$. 
For $i=1,2,\ldots,\ell$ choose two adjacent vertices $u_i$ and $v_i$ of $G_i$, 
where ${\rm deg}_{G_1}(u_1)=\Delta$. Let $G_{n,\Delta,\delta}$ be the graph 
obtained from $\bigcup_{i=1}^{\ell} G_i$ by deleting the edges
$u_iv_i$ for $i=2,3,\ldots,\ell-1$ and adding the edges $v_iu_{i+1}$ for 
$i=1,2,\ldots,\ell-1$. Clearly, $G_{n,\Delta,\delta}$ is a triangle-free graph
of order $n$, minimum degree $\delta$ and maximum degree $\Delta$. 
A calculation similar to the one in 
Theorem \ref{theo:maxdegree} shows that, with $N=n-\Delta+\delta$,  
\[ \mu(G_{n,\Delta,\delta}) 
\geq \frac{2}{3} \frac{N(N-1)}{n(n-1)} \frac{n+2\Delta}{\delta}  -8. \]
\end{proof}

Define $\mu'(n,\Delta,\delta)$ and $W'(n,\Delta,\delta)$ to be the maximum
average distance and 
maximum Wiener index, respectively, among all connected, triangle-free graphs of order $n$, minimum 
degree $\delta$ and maximum degree $\Delta$.
Theorem \ref{theo:maxdegree-triangle-free} shows that 
\[ \mu'(n,\Delta,\delta) = \frac{2}{3} \frac{(n-\Delta+\delta)(n-\Delta+\delta-1)}{n(n-1)} \frac{n+2\Delta}{\delta} +O(1)  \]
and
\[ W'(n,\Delta,\delta) 
 = \frac{2}{3} {n-\Delta+\delta \choose 2} \frac{n+2\Delta}{\delta}  +O(n^2).  \]

\section{Bounds for $C_4$-free graphs}
\label{section:C4-fre-graphs}

In this section we show that the bound in Theorem \ref{theo:maxdegree} can be 
improved significantly for graphs not containing a $4$-cycle.  If $\delta +1$ is
a prime power, then our bound is best possible in a sense specified later.

For the statement and proof of our bound for $C_4$-free graphs, 
we let $\delta \in \naturals$ be fixed, and for given $d\in \{ \delta, \Delta\}$ we 
define   
\begin{equation}  \label{eq:definition-of-epsilon_d}
\varepsilon_d := d\delta - 2\lfloor \frac{d}{2} \rfloor + 1. 
\end{equation} 
We need the following lemma, whose special case $d=\delta$ can be found, 
for example, in \cite{ErdPacPolTuz1989}.

\begin{lem} \label{la:N2-in-C4-free}
Let $G$ be a $C_4$-free graph of minimum degree $\delta$. If $v$ is a vertex of 
$G$ of degree at least $d$, then 
\[ |N_{\leq 2}(v)| \geq \varepsilon_d. \]
\end{lem}

\begin{proof}
Let $w_1, w_2,\ldots,w_d$ be the neighbours of $v$. Then each $w_i$ is adjacent to 
$v$ and at most one other vertex in $N(v)$. Hence each $w_i$ has a set  $U_i$ of 
at least $\delta-2$ neighbours in $N_2(v)$, and the sets $U_i$, $i=1,2,\ldots,d$ 
are disjoint. Hence 
$|N_{\leq 2}(v)| \geq |N[v]| + \sum_{i=1}^d |U_i| = 1+d + d(\delta-2) = d(\delta-1)+1$.
If $d$ is even, then the lemma follows. If $d$ is odd, the inequality is strict since
otherwise, if $|N_{\leq 2}(v)| = d(\delta-1)+1$, each of the $d$ neighbours of $v$
is adjacent to exactly one other neighbour of $v$, a contradiction to the handshake
lemma. 
\end{proof}

\begin{theorem} \label{theo:maxdegree-C4free}
Let $G$ be a connected graph of order $n$, minimum degree $\delta \geq 3$ and 
maximum degree $\Delta$. If $G$ does not contain $C_4$ as a subgraph, then 
\[ \mu(G) \leq  \frac{5}{3} 
\frac{N(N-1)}{n(n-1)} \frac{n+2\varepsilon_{\Delta}}{\varepsilon_{\delta}}  +8, \]
where $N=n-\varepsilon_{\Delta}+\varepsilon_{\delta}$, and 
$\varepsilon_{\Delta}$, $\varepsilon_{\delta}$ as defined 
in equation \ref{eq:definition-of-epsilon_d}.  
\end{theorem}

\begin{proof}
The proof is along similar lines as the proof of 
Theorem \ref{theo:maxdegree}, hence we do not give all the details.

We first find a $4$-packing $A$ of $G$ as follows. 
Choose a vertex $v_1$ of degree $\Delta$ in $G$ and let  $A= \{v_1\}$. 
If there exists a vertex $v_2$ of $G$ at distance exactly $5$ from $A$, then add 
$v_2$ to $A$. If there exists a vertex $v_3$ at distance exactly $5$ from $A$, then
add also $v_3$ to $A$. 
Repeat this step, i.e., successively add vertices at distance exactly $5$ from $A$  
until, after $k$ steps say, each vertex of $G$ is within distance $4$ of $A$.  
Then $A=\{v_1, v_2,\ldots, v_k\}$. 

By the construction of $A$ the sets $N_{\leq 2}(v_i)$, $i=1,2,\ldots,k$, are pairwise disjoint. 
For $i=1,2,\ldots, k$ let $T(v_i)$ be a tree with vertex set $N_{\leq 2}(v_i)$ 
that preserves the distances from  $v_i$. The disjoint union 
$\bigcup_{i=1}^k T(v_i)$ forms a subforest $T_1$ of $G$.  
It follows from the construction of $A$ that for every $i \in \{2,3,\ldots,k\}$
there exists an edge $e_i$ in $G$ joining a vertex in $T(v_i)$ to a vertex in $T(v_j)$
for some $j$ with $1 \leq j <i$.  Hence $T_2:=T_1 + \{e_2, e_3,\ldots, e_k\}$ is a subtree
of $G$. By the maximality of $A$ every vertex of $G$ not in $T_2$ is within distance 
$4$  from some vertex in $A$.  
We extend $T_2$ to a spanning tree $T$ of $G$ that preserves the distance
in $G$ to $A$.

It suffices to show that
\begin{equation} \label{eq:bound-on-mu(T)-C4free} 
\mu(T) \leq 
\frac{5}{3} \frac{(n-\varepsilon_{\Delta}+\varepsilon_{\delta})(n-\varepsilon_{\Delta}+\varepsilon_{\delta}-1)}{n(n-1)} \frac{n+2\varepsilon_{\Delta}}{\varepsilon_{\delta}}  +8.
\end{equation}  
For every vertex $u \in V (T)$ let $u_A$ be a vertex in $A$ closest to $u$ in $T$. 
We define a
weight function $c: V(T) \rightarrow \reals^{\geq 0}$ by 
\[ c(v) = | \{u \in V(T) \ | \ u_A = v\}|. \]

If $v\in A$, then every $u\in N_{\leq 2}(v)$ satisfies $u_A=v$.  
Lemma  \ref{la:N2-in-C4-free} thus yields that 
$c(v_1)\geq \varepsilon_{\Delta}$ and 
$c(v_i) \geq \varepsilon_{\delta}$ if $i \in \{2,3,\ldots,k\}$. 
If $v\in V(G)-A$, then $c(v)=0$. 
We also note that $\sum_{v \in V(T)} c(v)=n$. This yields 
$n = \sum_{v \in A}c(v) \geq \varepsilon_{\Delta} + (k-1)\varepsilon_{\delta}$, 
and so
\begin{equation} \label{eq:upper-bound-on-k-C4free}
k \leq \frac{n-\varepsilon_{\Delta}+ \varepsilon_{\delta}}{\varepsilon_{\delta}}. 
\end{equation}
Arguments similar to those in the proof of Theorem \ref{theo:maxdegree} 
show that 
\begin{equation} \label{eq:mu(T)-vs-mu_c(T)-C4free}
 \mu(T) \leq \mu_{c}(T) + 8. 
\end{equation} 
As in the proof of Theorem \ref{theo:maxdegree}, the graph 
$H:=T^5[A]$ is connected and satisfies 
\begin{equation}   \label{eq:mu_c(T)-vs-mu_c(H)-C4free} 
\mu_c(T) \leq 5 \mu_c(H). 
\end{equation}
We modify the weight function $c$ to obtain a new weight function $c'$ on $A$. 
Let 
\[ c'(v) = \left\{ \begin{array}{cc}
c(v) -\varepsilon_{\Delta} + \varepsilon_{\delta} & \textrm{if $v=v_1$,} \\
c(v)                  & \textrm{if $v \in A-\{v_1\}$}. 
           \end{array} \right.  \]
We have $c'(v) \geq \varepsilon_{\delta}$ for all $v \in A$. 
Let $N= \sum_{i=1}^k  c'(v_i)$, so $N = n-\varepsilon_{\Delta}+\varepsilon_{\delta}$. 
As in the proof of  
\ref{eq:mu_c-in-term-of-mu_c'}  in 
Theorem \ref{theo:maxdegree} shows that we can
express $\mu_c(H)$ in terms of $\mu_{c'}(H)$ as follows: 
\begin{equation}  \label{eq:mu_c-in-term-of-mu_c'-C4free} 
\mu_c(H) = \frac{N(N-1)}{n(n-1)} \mu_{c'}(H)  
             + \frac{2(\varepsilon_{\Delta}-\varepsilon_{\delta})}{n(n-1)} 
              \sum_{i=2}^k  c(v_i)d_H(v_1,v_i). 
\end{equation}             
Arguments similar to the proof of  \ref{eq:weighted-distance-of-v1-2} 
show that 
$\sum_{i=2}^k c(v_i)d_H(v_1,v_i) 
\leq \frac{N(N-\varepsilon_{\delta})}{2\varepsilon_{2\delta}}$ 
and thus 
\begin{equation}  \label{eq:weighted-distance-of-v1-3-C4free}
\sum_{v \in A-\{v_1\}} c(v)d_H(v_1,v) 
 \leq \frac{N(N-1)}{2\varepsilon_{\delta}}.
\end{equation}
and
\begin{equation}   \label{eq:bound-on-mu_c'(H)-C4free}
\mu_{c'}(H) \leq \frac{N+\varepsilon_{\delta}}{3\varepsilon_{\delta}}. 
\end{equation}
Substituting  \ref{eq:weighted-distance-of-v1-3-C4free} and 
\ref{eq:bound-on-mu_c'(H)-C4free} into \ref{eq:mu_c-in-term-of-mu_c'-C4free} 
yields that 
\begin{eqnarray}
\mu_{c}(H) 
& \leq & \frac{N(N-1)}{n(n-1)} \frac{N+\varepsilon_{\delta}}{3\varepsilon_{\delta} }
    + \frac{\varepsilon_{\Delta}-\varepsilon_{\delta}}{n(n-1)} 
      \frac{N(N-1)}{\varepsilon_{\delta}}  \nonumber \\
      & = & \frac{N(N-1)}{n(n-1)} 
        \frac{n+2\varepsilon_{\Delta}-\varepsilon_{\delta}}{3\varepsilon_{\delta}} 
                                              \nonumber \\
      & < & \frac{N(N-1)}{n(n-1)} \frac{n+2\varepsilon_{\Delta}}{3\varepsilon_{\delta}}. 
      \label{eq:bound-on-mu(H)-C4free}
\end{eqnarray}
Combining the inequalities  \ref{eq:mu(T)-vs-mu_c(T)-C4free}, 
\ref{eq:mu_c(T)-vs-mu_c(H)-C4free}  
and  \ref{eq:bound-on-mu(H)-C4free} we obtain that 
\begin{eqnarray*}
\mu(T) & \leq & \mu_c(T) + 8 \\
 & \leq & 5 \mu_c(H) + 8 \\
 & \leq & \frac{5}{3} \frac{N(N-1)}{n(n-1)} \frac{n+2\varepsilon_{\Delta}}{\varepsilon_{\delta}}  +8, 
\end{eqnarray*}
which is \ref{eq:bound-on-mu(T)-C4free}, and so the bound in the theorem follows. 
\end{proof}

The following theorem shows that the bound in Theorem \ref{theo:maxdegree-C4free} 
is not far from best possible if $\delta+1$ is a prime power. Our construction 
is based on the graph $H_q$ described below, constructed independently
by \cite{ErdRen1962} and Brown \cite{Bro1966}, and a modification $H_q'$, first described in \cite{ErdPacPolTuz1989}. 

Let $q$ be a prime power. Let $GF(q)$ be the field of order $q$ and let $GF(q)^3$ 
be the $3$-dimensional vector space over $GF(q)$ whose vectors are the 
triples of elements of $GF(q)$. We define $H_q$ to be the graph whose vertices
are the $1$-dimensional subspaces of $GF(q)^3$, where two vertices are adjacent
if, as subspaces, they are orthogonal. Using elementary linear algebra it is easy 
to verify that $H_q$ has $q^2+q+1$ vertices, that its vertices have degree either 
$q$ or $q+1$, and that $H_q$ is $C_4$-free and connected.  

\cite{ErdPacPolTuz1989} described a 
modification $H_q'$ of the graph $H_q$, obtained as follows. 
Choose a vertex $z$ of degree $q$ in $H_q$ and let $u,v$ be two distinct neighbours
of $z$. Delete $z$ and all edges joining a neighbour of $u$ to a neighbour of $v$. 
The resulting graph $H_q'$ is connected, $C_4$-free and has $q^2+q$ vertices, its
minimum degree is at least $q-1$, and $d_{H_q'}(u,v) \geq 4$. 

\begin{figure}
\begin{center}
\hspace*{-5mm}
\begin{tikzpicture}
  [scale=0.38,inner sep=1mm, 
   vertex/.style={circle,thick,draw}, 
   thickedge/.style={line width=2pt}] 
    \begin{scope}[>=triangle 45] 
    \node[vertex] (a2) at (0.5,2) [fill=black] {};
    \node[vertex] (a3) at (4,2) [fill=black] {};
    \node[vertex] (a4) at (8,2) [fill=black] {};
    \node[vertex] (a5) at (10,2) [fill=black] {};
    \node[vertex] (a6) at (14,2) [fill=black] {};
    \node[vertex] (a7) at (16,2) [fill=black] {};
    \node[vertex] (a8) at (20,2) [fill=black] {};
    \node[vertex] (a9) at (22,2) [fill=black] {};
    \node[vertex] (a10) at (26,2) [fill=black] {};    
    \draw[black, very thick] (a2)--(a3)  (a4)--(a5) (a6)--(a7) (a8)--(a9);  

  \draw[black, thick] (-0.5,4) circle [radius =2.7];
  \draw[black, thick] (-0.5,0) circle [radius =2.7];
    \draw[black, thick] (6,2) circle [radius =2.7]; 
    \draw[black, thick] (12,2) circle [radius =2.7]; 
    \draw[black, thick] (18,2) circle [radius =2.7];       
    \draw[black, thick] (24,2) circle [radius =2.8];                 
    \node[above] at (1.4,2) {$v_{1}$};  
    \node[above] at (7.4,2) {$v_{2}$}; 
    \node[above] at (13.4,2) {$v_{3}$}; 
    \node[above] at (19.4,2) {$v_{4}$}; 
    \node[above] at (25.4,2) {$v_{5}$};  
    \node[above] at (4.2,2) {$u_{2}$}; 
    \node[above] at (10.2,2) {$u_3$}; 
    \node[above] at (16.2,2) {$u_{4}$}; 
    \node[above] at (22.2,2) {$u_{5}$};    
    \node[below] at (6.05,1.6) {$H_{q}'$}; 
    \node[below] at (12.05,1.6) {$H_{q}'$}; 
    \node[below] at (18.05,1.6) {$H_{q}'$};           
    \node[below] at (24.05,1.6) {$H_{q}'$};               
    \node[above] at (-0.5,4.6) {$H_{q}$}; 
    \node[above] at (-0.5,-1.6) {$H_{q}$}; 
    
    \node[left] at (-2.2,2) {$G_1$}; 
    \node[below] at (6.05,-0.7) {$G_2$}; 
    \node[below] at (12.05,-0.7) {$G_3$}; 
    \node[below] at (18.05,-0.7) {$G_4$};           
    \node[below] at (24.05,-0.7) {$G_5$};      
    \end{scope}
\end{tikzpicture}
\end{center}
\caption{The graph $G_{k,\ell,q}$ for $k=2$ and $\ell=5$.}
\label{fig:sharpness-ex-C4free}
\end{figure}
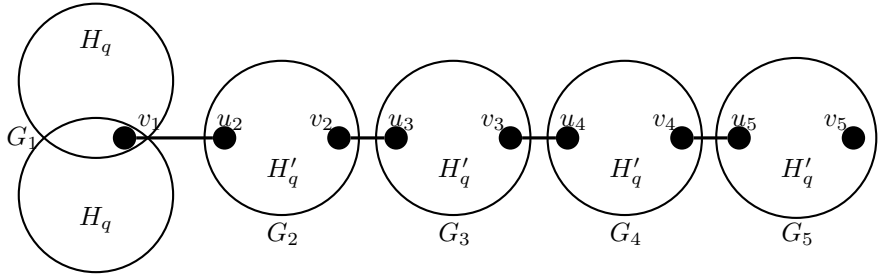

\begin{theorem}   \label{theo:C4-free-construction}
Let $\delta \in \naturals$ such that $\delta\geq 3$ and $\delta+1$ is a prime power. 
Let $\Delta\in \naturals$ such that $\Delta-1$ is a positive multiple of $\delta+2$. 
Let $n\in \naturals$ such that $n-(\Delta-1)(\delta+1)-1$ is a positive multiple
of $(\delta+1)(\delta+2)$. Then there exists a $C_4$-free graph $G$ of order $n$, 
maximum degree $\Delta$ and minimum degree at least $\delta$ which satisfies
\[   \mu(G) > 
\frac{5}{3} \frac{M(M-1)}{n(n-1)} \frac{n+2\theta_{\Delta}}{\theta_{\delta}} 
      - 13, \]
where $\theta_{\Delta} = (\Delta-1)(\delta+1)+1$,  
$\theta_{\delta} = (\delta+2)(\delta+1)$, 
and $M=n-\theta_{\Delta}+ \theta_{\delta}$.      
\end{theorem}

\begin{proof}
Let $q=\delta+1$. By the assumptions on $\Delta$ and $n$ we can find 
$k, \ell \in \naturals$ with $\ell\geq 2$ such that $\Delta=k(q+1)+1$ and 
$n = (k+\ell-1)(q^2+q)+1$. 

We construct a graph $G_1$ by taking $k$ disjoint copies of $H_q$, choosing
a vertex of degree $q+1$ in each copy, and then identifying these $k$ vertices to
a new vertex $v_1$. Clearly, $n(G_1) =k(q^2+q)+1$ and ${\rm deg}_{G_1}(v_1)=k(q+1)$. 
For $i=2,3\ldots,\ell$ let $G_i$ be a copy of the graph $H_q'$. Let $u_i$ and $v_i$ 
denote the vertices of $G_i$ corresponding to $u$ and $v$, respectively, of $H_q'$. 
Let $G_{k,\ell,q}$ be the  graph obtained from the disjoint union 
$\bigcup_{i=1}^{\ell} G_i$ by adding the edges $v_iu_{i+1}$ for $i=1,2,\ldots,\ell-1$. 
A sketch of the graph $G_{2,5,q}$ is shown in Figure \ref{fig:sharpness-ex-C4free}. 
Then  
$G_{k,\ell,q}$ has order $n$, maximum degree $\Delta$ and minimum degree at least 
$\delta$. We define 
$\theta_{\Delta} = (\Delta-1)(\delta+1)+1$,  
$\theta_{\delta} = (\delta+2)(\delta+1)$, 
and $M=n-\theta_{\Delta}+ \theta_{\delta}$.

We now bound the average distance of $G_{k,\ell,q}$ from below in terms
of $n$, $\theta_{\Delta}$ and $\theta_{\delta}$. 
Let $V_i:=V(G_i)$ for $i=1,2,\ldots,\ell$. Then 
$|V_1|= \theta_{\Delta}$ and $|V_i|= \theta_{\delta}$ for 
$i=2,3,\ldots,\ell$. 

For our lower bound we only count the distances
between pairs $x,y$ with either $x\in V_i$ and $y\in V_j$ where $2\leq i<j\leq \ell$, or
$x \in V_1$ and $y \in V_j$ for $j=2,3,\ldots,\ell$, ignoring all
other pairs of vertices. Clearly, if $x \in V_i$, $y\in V_j$ and $i<j$, then 
$d(x,y) \geq 5(j-i)-4$. If $x \in V_1$ and $y \in V_j$, then 
$d(x,y) \geq 5j-9$. Hence 
\begin{eqnarray*}
W(G_{k,\ell,q}) & > & \sum_{2\leq i < j \leq \ell} 
           \sum_{x\in V_i, y \in V_j} d(x,y) + 
           \sum_{j=2}^{\ell} \sum_{x\in V_1, y \in V_j} d(x,y) \\
    & \geq &  \sum_{2\leq i < j \leq \ell} |V_i| \, |V_j|  \big( 5(j-i)-4 \big) 
          + \sum_{j=2}^{\ell} |V_1| \, |V_j| \big( 5j-9\big) \\
     & = & \theta_{\delta}^2 \sum_{2\leq i < j \leq \ell}  \big( 5(j-i)-4 \big) 
          + \theta_{\delta} \theta_{\Delta} \sum_{j=2}^{\ell}  \big( 5j-9\big).
\end{eqnarray*}
Straightforward calculations show that 
$\sum_{2\leq i < j \leq \ell}  \big( 5(j-i)-4 \big)  
   = \frac{1}{6}(5\ell^3-27\ell^2 + 46\ell - 24)
   \geq \frac{1}{6}(5\ell^3 - 27 \ell^2)$  
and $\sum_{j=2}^{\ell}  \big( 5j-9\big) = \frac{1}{2}(5\ell^2 -13\ell+8)
 \geq \frac{1}{2}(5\ell^2-13\ell)$. 
Substituting these values  we obtain
\begin{eqnarray*}
W(G_{k,\ell,q}) & > &  
  \frac{1}{6} \theta_{\delta}^2(5\ell^3 - 27 \ell^2)   
          + \frac{1}{2}\theta_{\delta} \theta_{\Delta} (5\ell^2-13\ell) \\
   & = & \frac{5}{6}\big(\theta_{\delta}^2\ell^3 
   + 3 \theta_{\delta}\theta_{\Delta}\ell^2 \big) 
    - \frac{9}{2}\theta_{\delta}^2 \ell^2 
       - \frac{13}{2} \theta_{\Delta} \theta_{\delta} \ell.
\end{eqnarray*}
Since $\ell \theta_{\delta} = M$, we have 
$\theta_{\delta}^2\ell^3 
   + 3 \theta_{\delta}\theta_{\Delta} \ell^2 
 = \frac{M^3}{\theta_{\delta}} 
       + 3 \frac{\theta_{\Delta} M^2}{\theta_{\delta}}
\geq \frac{M(M-1)(n+2\theta_{\Delta})}{\theta_{\delta}}$. 
Also 
$\frac{9}{2}\theta_{\delta}^2 \ell^2 
       + \frac{13}{2} \theta_{\Delta} \theta_{\delta} \ell
  < \frac{13}{2}\theta_{\delta}\ell ( \theta_{\delta} \ell+ \theta_{\Delta}) 
=\frac{13}{2}M(M+\theta_{\Delta}) 
  = \frac{13}{2}
(n-\theta_{\Delta} + \theta_{\delta})
(n+\theta_{\delta}) 
  \leq \frac{13}{2} (n-1)n$.  
Dividing the above lower bound on $W(G_{k,\ell,q})$ by ${n \choose 2}$ thus yields 
\[ \mu(G_{k,\ell,q}) > \frac{5}{3} \frac{M(M-1)}{n(n-1)} 
      \frac{n+2\theta_{\Delta}}{\theta_{\delta}} 
      - 13, \]
as desired. 
\end{proof}

To see that Theorem \ref{theo:maxdegree-C4free} is not far from best possible, even
for very large maximum degree, assume that $\delta, \Delta \in \naturals$ 
satisfy the hypothesis of Theorem \ref{theo:C4-free-construction}, and that moreover
$\Delta = cn$ for some $c \in \naturals$ with $0<c <\frac{1}{\delta+1}$ (note that
by Lemma \ref{la:N2-in-C4-free} $c$ cannot be greater than $\frac{1}{\delta+1}$).
Then the leading term in the upper bound in Theorem \ref{theo:maxdegree-C4free} is 
$\frac{5}{3} \frac{(1-c\delta)^2(1+2c\delta)}{\delta^2 - 2\lfloor \delta/2 \rfloor +1} n$
while the leading term in the  lower bound in Theorem \ref{theo:C4-free-construction} is
$\frac{5}{3} \frac{(1-c(\delta+1)^2(1+2c(\delta+1))}{\delta^2+\delta-2} n$.
It is easy to see that the ratio of the two coefficients of $n$ approaches $1$
as $\delta$ gets large. So the larger $\delta$, the closer the bound
in Theorem \ref{theo:maxdegree-C4free} to being sharp.

\section{Open Problem}
\label{section:open-problems}

In this paper we gave three instances in which bounds on the Wiener index or the
average distance for graphs of given order and minimum degree can be improved
significantly if the graph contains a vertex of large degree. The same question
can be asked for bounds in terms of other graph parameters. For example it
was shown by \cite{Ple1984} that the cycle is the unique graph maximising the
Wiener index among all $2$-edge-connected or $2$-connected graphs of given order,
and corresponding results for graphs of higher connectivity were given in \cite{DanMukw2009}.
Can these bounds be improved significantly if the graph has a vertex of large
degree?

\begin{problem}
Determine the maximum Wiener index among all $2$-connected graphs of given order 
and maximum degree.
\end{problem} 

\clearpage
\nocite{*}
\bibliographystyle{abbrvnat}
\bibliography{sample-dmtcs.bib}


\end{document}